\documentclass[smallextended]{svjour3}       

\smartqed  

\usepackage{latexsym}
\usepackage{amsmath}
\usepackage{amssymb}

\usepackage{epsfig}

\usepackage{enumerate}

\newcommand{\cI}{{\mathcal I}}
\newcommand{\cN}{{\mathcal N}}
\newcommand{\cQ}{{\mathcal Q}}
\newcommand{\cR}{{\mathcal R}}
\newcommand{\cS}{{\mathcal S}}
\newcommand{\cT}{{\mathcal T}}

\usepackage{color}

\begin{document}

\title{Caterpillars on three and four leaves are sufficient to reconstruct normal networks \thanks{The authors were supported by the New Zealand Marsden Fund.}}

\titlerunning{Reconstructing normal networks from caterpillars}        



\author{Simone Linz \and Charles Semple}

\institute{Simone Linz \at
School of Computer Science, University of Auckland, Auckland, New Zealand \\
\email{s.linz@auckland.ac.nz}
\and
Charles Semple \at
School of Mathematics and Statistics, University of Canterbury, Christchurch, New Zealand \\
\email{charles.semple@canterbury.ac.nz}
}

%

\date{Received: date / Accepted: date}

%

\date{\today}

\maketitle

\begin{abstract}
While every rooted binary phylogenetic tree is determined by its set of displayed rooted triples, such a result does not hold for an arbitrary rooted binary phylogenetic network. In particular, there exist two non-isomorphic rooted binary temporal normal networks that display the same set of rooted triples. Moreover, without any structural constraint on the rooted phylogenetic networks under consideration, similarly negative results have also been established for binets and trinets which are rooted subnetworks on two and three leaves, respectively. 
Hence, in general, piecing together a rooted phylogenetic network from such a set of small building blocks appears insurmountable. In contrast to these results, in this paper, we show that a rooted binary normal network is determined by its sets of displayed caterpillars (particular type of subtrees) on three and four leaves. The proof is constructive and realises a polynomial-time algorithm that takes the sets of caterpillars on three and four leaves displayed by a rooted binary normal network and, up to isomorphism, reconstructs this network.
\keywords{Normal networks \and rooted triples \and quads}
\subclass{05C85 \and 68R10}
\end{abstract}


\section{Introduction}

Rooted phylogenetic networks are a generalisation of rooted phylogenetic trees that allow for the representation of non-treelike processes such as hybridisation and lateral gene transfer. While many structural properties of rooted phylogenetic networks have recently been described that have led to a classification of such networks into several well-studied network classes (e.g.\ normal~\cite{willson10}, tree-child~\cite{cardona09}, and tree-based~\cite{francis15}), there remains a lack of practical algorithms to reconstruct them.

In the context of reconstructing rooted phylogenetic trees, supertree methods collectively provide fundamental tools for reconstructing and analysing rooted phylogenetic trees. In general, these classical and commonly-used methods take as input a collection of smaller rooted phylogenetic trees on overlapping leaf sets and output a parent tree (supertree) that `best' represents the entire input collection. For practical reasons, most of these methods do not require the input collection to be consistent. Nevertheless, the property typically underlying any supertree method for reconstructing rooted phylogenetic trees is the following well-known theorem (see, for example, \cite{aho81,semple03}).
\begin{theorem}
Let $\cR$ be the set of rooted triples displayed by a rooted binary phylogenetic $X$-tree $\cT$. Then, up to isomorphism,
\begin{enumerate}[{\rm (i)}]
\item $\cT$ is the unique rooted binary phylogenetic $X$-tree whose set of displayed rooted triples is $\cR$, and

\item $\cT$ can be reconstructed from $\cR$ in polynomial time.
\end{enumerate}
\label{triple}
\end{theorem}
\noindent For an excellent review of supertree methodology, see~\cite{bin04}. As an initial step towards developing supertree-type methods for reconstructing and analysing rooted phylogenetic networks, we would like analogues of Theorem~\ref{triple} for rooted phylogenetic networks.

Gambette and Huber~\cite{gambette12} established that rooted binary level-one networks, that is, rooted binary phylogenetic networks whose underlying cycles are vertex disjoint, are determined by their sets of displayed rooted triples provided each underlying cycle has length at least four. However, there exist two non-isomorphic rooted binary level-two networks that have the same set of displayed rooted triples~\cite[Fig.~11]{gambette12}. This begs the question whether or not displayed subtrees on more than three leaves are sufficient to determine rooted phylogenetic networks in general. While Willson~\cite{willson11} has shown that rooted binary regular networks, which include the class of rooted binary normal networks, on $n$ leaves can be determined and reconstructed (in polynomial time) from their sets of displayed rooted phylogenetic trees on $n$ leaves, arbitrary rooted binary phylogenetic networks cannot be determined in this way, even if branch lengths are considered~\cite[Fig.~3]{pardi15}. As an interesting aside, Francis and Moulton~\cite[Theorem 3.5]{francis18} have shown that  rooted binary tree-child networks are determined by their sets of embedded spanning trees which, importantly, are not necessarily rooted phylogenetic trees.

Partly due to the aforementioned negative deterministic results, recent studies have investigated whether or not rooted phylogenetic networks are determined by their embedded subnetworks like binets~\cite{huber17,vanIersel17} and trinets~\cite{huber17,huber12,vanIersel13}, that is, rooted phylogenetic networks on two and three leaves, respectively. It has been shown that trinets determine rooted binary level-two and rooted binary tree-child networks~\cite{vanIersel13}.
Furthermore, binets determine the number of vertices in a rooted phylogenetic network whose in-degree is at least two but do not contain enough information to determine the structural properties of a rooted phylogenetic network even for restricted network classes~\cite{vanIersel17}. Lastly, for an arbitrary rooted binary phylogenetic network $\cN$ on $n$ leaves, Huber et al.~\cite{huber14} have considered larger subnetworks and shown that, even if for each $n'\in\{1, 2, \ldots, n-1\}$ all embedded subnetworks of $\cN$ on $n'$ leaves are given, $\cN$ cannot necessarily be determined by the resulting set of subnetworks. 

In this paper, we return to the simpler tree-like building blocks of small size for reconstructing rooted phylogenetic networks. In particular, the main result of the paper establishes that rooted binary normal networks are determined by their sets of caterpillars (particular type of subtrees) on three and four leaves and that they can be reconstructed from these sets in polynomial time.

To formally state the main result, we need some notation and terminology.  Throughout the paper, $X$ will always denote a non-empty finite set. A {\em rooted binary phylogenetic network $\cN$ on $X$} is a rooted acyclic directed graph with no parallel arcs satisfying the following three properties:
\begin{enumerate}[(i)]
\item the (unique) root has in-degree zero and out-degree two;

\item a vertex of out-degree zero has in-degree one, and the set of vertices with out-degree zero is $X$; and

\item all other vertices either have in-degree one and out-degree two, or in-degree two and out-degree one.
\end{enumerate}
For technical reasons, if $|X|=1$, then we additionally allow $\cN$ to consist of the single vertex in $X$. The vertices of $\cN$ of out-degree zero are called {\em leaves}, and so $X$ is referred to as the {\em leaf set} of $\cN$. Furthermore, vertices of in-degree one and out-degree two are {\em tree vertices}, while vertices of in-degree two and out-degree one are {\em reticulations}. Arcs directed into a reticulation are called {\em reticulation arcs}, all other arcs are {\em tree arcs}. A {\em rooted binary phylogenetic $X$-tree} is a rooted binary phylogenetic network on $X$ with no reticulations. To ease reading, for the rest of the paper, we will refer to rooted binary phylogenetic networks and rooted binary phylogenetic trees as phylogenetic networks and phylogenetic trees, respectively, as all such networks and trees are rooted and binary.

Let $\cN_1$ and $\cN_2$ be two phylogenetic networks on $X$ with vertex and arc sets $V_1$ and $E_1$, and $V_2$ and $E_2$, respectively. We say $\cN_1$ is {\em isomorphic} to $\cN_2$ if there is a bijection $\varphi: V_1\rightarrow V_2$ such that $\varphi(x)=x$ for all $x\in X$, and $(u, v)\in E_1$ if and only if $(\varphi(u), \varphi(v))\in E_2$ for all $u, v\in V_1$.

\begin{figure}
\center
\scalebox{1.2}{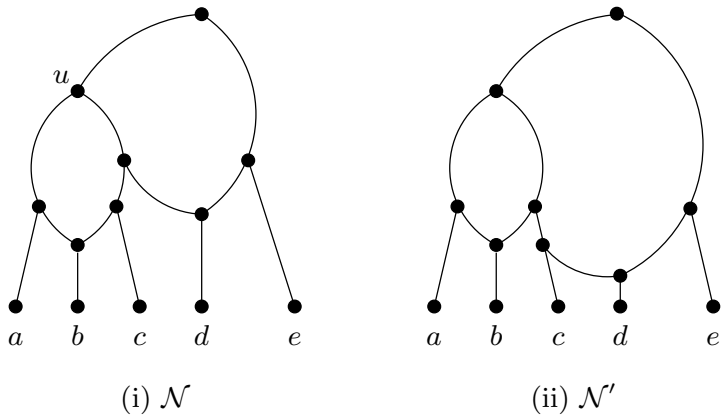}
\caption{Two normal networks on $\{a, b, c, d, e\}$ with the same sets of displayed triples.}
\label{fig:triples}
\end{figure}

Let $\cN$ be a phylogenetic network on $X$. A reticulation arc $(u, v)$ of $\cN$ is a {\em shortcut} if $\cN$ has a directed path from $u$ to $v$ avoiding $(u, v)$. We say $\cN$ is {\em tree-child} if every non-leaf vertex is the parent of a tree vertex or a leaf. Moreover, $\cN$ is {\em normal} if it is tree-child and has no shortcuts. An example of two normal networks is shown in Fig.~\ref{fig:triples}, where, as with all figures in this paper, arcs are directed down the page.

\begin{figure}
\center
\scalebox{1.2}{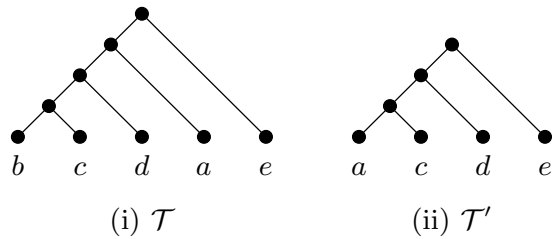}
\caption{Two caterpillars.}
\label{fig:caterpillar}
\end{figure}

Let $\cT$ be a phylogenetic $X$-tree. We call $\cT$ a {\em caterpillar} if we can order its leaf set $X$, say $x_1, x_2, \ldots, x_n$, so that the parents of $x_1$ and $x_2$ are the same and, for all $i\in \{2, 3, \ldots, n-1\}$, we have that $(p_{i+1}, p_i)$ is an arc in $\cT$, where $p_{i+1}$ and $p_i$ are the parents of $x_{i+1}$ and $x_i$, respectively. We denote such a caterpillar $\cT$ by $(x_1, x_2, \ldots, x_n)$ or, equivalently, $(x_2, x_1, x_3, \ldots, x_n)$ where $x_1$ and $x_2$ have been interchanged. As an example, the two phylogenetic trees $\cT$ and $\cT'$ in Fig.~\ref{fig:caterpillar} are caterpillars on five and four leaves, respectively. Here, $\cT$ is denoted by $(b, c, d, a, e)$ and $\cT'$ is denoted by $(a, c, d, e)$. For a caterpillar $\cT$ on $X$, we say that $\cT$ is a {\em triple} if $|X|=3$ and we say that $\cT$ is a {\em quad} if $|X|=4$. While we will denote quads as $4$-tuples, we will denote the triple $(x_1, x_2, x_3)$ by $x_1x_2|x_3$ in keeping with standard notation (e.g.\ see~\cite{semple03}). Note that triples are also referred to as rooted triples in the literature.

Now let $\cN$ be a phylogenetic network on $X$, and let $\cT$ be a phylogenetic $X'$-tree, where $X'$ is a non-empty subset of $X$. Then $\cN$ {\em displays} $\cT$ if $\cT$ can be obtained from $\cN$ by deleting arcs and vertices, and suppressing any resulting vertices of in-degree one and out-degree one. To illustrate, consider Fig.~\ref{fig:triples}. The caterpillar $(b, c, d, a, e)$ is displayed by $\cN$, but the caterpillar $(a, c, d, e)$ is not displayed by $\cN$. Hence, $(a, c, d, e)$ is not an element of the set of quads that are displayed by $\cN$. The main result of this paper is the following theorem.

\begin{theorem}
Let $\cR$ and $\cQ$ be the sets of triples and quads displayed by a normal network $\cN$ on $X$, respectively. Then, up to isomorphism, 
\begin{enumerate}[{\rm (i)}]
\item $\cN$ is the unique normal network on $X$ whose sets of displayed triples and quads are $\cR$ and $\cQ$, and

\item $\cN$ can be reconstructed from $\cR$ and $\cQ$ in $O(1)$ time if $|X|\in \{1, 2\}$ and $O(|X|^4(|X||\cR|+|X|^2|\cR||\cQ|))$ time, that is in $O(|X|^{13})$ time, if $|X|\ge 3$.
\end{enumerate}
\label{main}
\end{theorem}
\noindent Note that, in the running time of Theorem~\ref{main}, if $|X|\ge 3$ and $|\cQ|\ge 1$, then $|X||\cR|$ is dominated by $|X|^2|\cR||\cQ|$, and so the running time is $O(|X|^6|\cR||\cQ|)$. It is easily checked that if a normal network has at least one reticulation and four leaves, at least two reticulations, or at least five leaves, then it has at least one quad.

It is natural to ask whether Theorem~\ref{main} can be strengthened. In particular, are (i) normal networks determined by their displayed triples and are (ii) tree-child networks determined by their displayed triples and quads? For both (i) and (ii), the answer, in general, is no. To see this, first consider the two normal networks on $\{a, b, c, d, e\}$ shown in Fig.~\ref{fig:triples}. Here, both networks display the same set of triples but are not isomorphic. For (ii), consider the three tree-child networks shown in Fig~\ref{fig:tree-child}. All three networks display the same sets of triples and quads, but no two networks are isomorphic. In fact, all three networks display the same set of phylogenetic $X$-trees, where $X$ is the leaf set of each of the three tree-child networks.

\begin{figure}
\center
\scalebox{1.2}{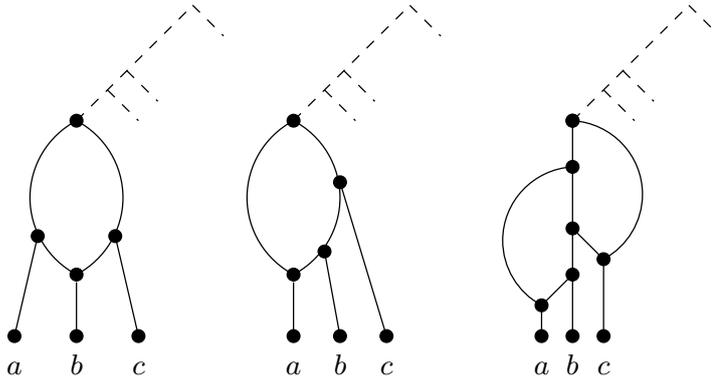}
\caption{Three tree-child networks on $\{a, b, c\}$ (solid edges) that display the same set of triples. By adding additional leaves, say $x_1, x_2, \ldots, x_n$, as indicated by the dashed edges, the three networks can be extended to tree-child networks of arbitrary size that display the same set of triples and phylogenetic trees on $\{a, b, c\}\cup \{x_1, x_2, \ldots, x_n\}$.}
\label{fig:tree-child}
\end{figure}

The paper is organised as follows. The next section contains some preliminaries that are used throughout the paper. The proof of Theorem~\ref{main} relies on being able to recognise so-called cherries and reticulated cherries, certain structures involving two leaves, using only triples and quads. This recognition is established in Section~\ref{key}. The proof of Theorem~\ref{main} as well as the associated algorithm for reconstructing a normal network from its triples and quads are given in Section~\ref{proof}. In the last section, we consider the class of temporal normal networks, and highlight how the approach taken in Theorem~\ref{main} can be simplified for reconstructing such networks from their sets of triples and quads.

Lastly, normal networks are a rich class of phylogenetic networks. Thus, given the negative results mentioned earlier in the introduction, it is a little surprising that they are determined by their sets of triples and quads as Theorem~\ref{main} establishes. Nevertheless, knowing that they are provides impetus for developing a supertree-type method for constructing normal networks. To this end, an intermediate step is to develop an algorithm for deciding if, given a set of triples and quads on overlapping leaf sets, there is a normal network that displays each caterpillar in the given set.

\section{Preliminaries}\label{sec:prelim}

In this section, we state some further notation and terminology used in the paper. We begin by noting that an immediate consequence of the definition is that a phylogenetic network $\cN$ is tree-child if and only if, for every vertex $u$ of $\cN$, there is a directed path from $u$ to a leaf, $\ell$ say, in which each vertex, except $\ell$ and possibly $u$, is a tree vertex. In particular, every vertex in a normal network has this property. We refer to such a path as a {\em tree path (for $u$)}. Note that if $u$ is a leaf, then the path consisting of just $u$ is a tree path for $u$.

\subsection{Embeddings} Let $\cN$ be a phylogenetic network on $X$, and let $\cT$ be a phylogenetic $X'$-tree, where $X'\subseteq X$. An equivalent and convenient way to view the notion of display is as follows. The {\em root extension} of $\cT$ is obtained by adjoining a new vertex, $u$ say, to the root of $\cT$ via a new arc directed away from $u$. It is easily seen that $\cN$ displays $\cT$ if and only if a subdivision, $\cS$ say, of either $\cT$ or the root extension of $\cT$ can be obtained from $\cN$ by deleting arcs and non-root vertices, in which case, the roots of $\cS$ and $\cN$ coincide. This equivalence is freely used throughout the paper. We refer to $\cS$ as an {\em embedding} of $\cT$ in $\cN$ and, for convenience, sometimes view $\cS$ as the arc set of $\cS$.  

Let $\cS$ be an embedding of a phylogenetic tree $\cT$ in $\cN$, and let $(u, v)$ be an arc in $\cN$. If $(u, v)$ is an arc in $\cS$, then $\cS$ {\em uses} $(u, v)$; otherwise, it {\em avoids} $(u, v)$. Analogous terminology holds for the vertices in $\cN$.

\subsection{Cherries and reticulated cherries} For each leaf $x$ of a phylogenetic network, we denote the parent of $x$ by $p_x$. 
Now let $a$, $b$, and $c$ be distinct leaves of a phylogenetic network $\cN$. If $p_a=p_b$, then $\{a, b\}$ is a {\em cherry} of $\cN$. Furthermore, if $p_b$ is a reticulation and $(p_a, p_b)$ is an arc, then $\{a, b\}$ is a {\em reticulated cherry} of $\cN$ in which $b$ is the {\em reticulation leaf}. Moreover, in this case, we denote the parent of $p_b$ that is not $p_a$ by $g_b$. As an example, in Fig.~\ref{fig:triples}, $\{b, c\}$ is a reticulated cherry in which $b$ is the reticulation leaf in $\cN$. However, in the same figure, $\{b, c\}$ is not a reticulated cherry in $\cN'$.

The next lemma is well known for tree-child networks (for example, see~\cite{bor16}). The restriction to normal  networks is immediate. We will use it freely throughout the paper.

\begin{lemma}
Let $\cN$ be a normal network on $X$, where $|X|\ge 2$. Then $\cN$ has either a cherry or a reticulated cherry. 
\label{cherry3}
\end{lemma}

\subsection{Cluster and visibility sets} Let $\cN$ be a phylogenetic network on $X$, and let $u$ be a vertex of $\cN$. The {\em cluster set} of $u$, denoted $C_u$, is the subset of $X$ consisting exactly of each leaf $\ell$ in $X$ for which there is a directed path from $u$ to $\ell$. Furthermore, the {\em visibility set} of $u$, denoted $V_u$, is the subset of $X$ consisting exactly of each leaf $\ell$ in $X$ for which every directed path from the root of $\cN$ to $\ell$ traverses $u$. Observe that $V_u\subseteq C_u$ and that, if there is a tree path from $u$ to a leaf $\ell$, then $\ell\in V_u$. In particular, if $\cN$ is normal, then $V_u$, and thus $C_u$, is non-empty. To illustrate, consider the vertex $u$ in Fig~\ref{fig:triples}(i). The cluster and visibility sets of $u$ are $C_u=\{a, b, c, d\}$ and $V_u=\{a, b, c\}$, respectively. Note that $d\not\in V_u$ as there is a directed path from the root of $\cN$ to $d$ avoiding $u$.

\section{Recognising Cherries}
\label{key}

The key idea in the proof of Theorem~\ref{main} is recognising cherries and reticulated cherries in a normal network $\cN$ using only the triples and quads displayed by $\cN$. In this section, we establish the lemmas for doing this. We begin by recognising cherries.

\begin{lemma}
Let $\cN$ be a normal network on $X$, where $|X|\ge 3$, and let $\cR$ be the set of triples displayed by $\cN$. Let $\{a, b\}\subseteq X$. Then $\{a, b\}$ is a cherry of $\cN$ if and only if $\{a, b\}$ satisfies the following property: if $xy|z\in \cR$ and $\{a, b\}\subseteq \{x, y, z\}$, then $\{a, b\}=\{x, y\}$.
\label{cherry1}
\end{lemma}

\begin{proof}
If $\{a, b\}$ is a cherry of $\cN$, then it is easily checked that $\{a, b\}$ satisfies the property in the statement of the lemma. Now suppose that $\{a, b\}$ satisfies this property. First assume $p_a$ is a reticulation. Let $u$ and $u'$ be the two parents of $p_a$. Since $\cN$ is normal, there are two distinct elements, say $\ell$ and $\ell'$, in $X-\{a\}$ such that $\ell$ is at the end of a tree path for $u$ and $\ell'$ is at the end of a tree path for $u'$. If $\ell\neq b$, then either $a\ell|b\in \cR$ or $b\ell|a\in \cR$, a contradiction. So $\ell=b$, but then $a\ell'|b\in\cR$, another contradiction. Thus $p_a$ is a tree vertex or the root of $\cN$. If $p_a$ is the root of $\cN$, then, as $|X|\ge 3$, there exists a triple $b\ell|a\in \cR$, where $\ell\in X-\{a, b\}$, a contradiction. So $p_a$ is a tree vertex.

Let $v$ be the child of $p_a$ that is not $a$. If $b\not\in C_v$, then $a\ell|b\in \cR$, where $\ell\in C_v$, a contradiction. Therefore $b\in C_v$. If $|C_v| > 1$, then there is an element $\ell\in C_v-\{b\}$ such that $b\ell|a\in\cR$. This last contradiction implies that $C_v=\{b\}$, and so $\{a, b\}$ is either a cherry or a reticulated cherry of $\cN$ with reticulation leaf $b$. If the latter, then $p_b=v$ is a reticulation. In this case, let $g_b$ be the parent of $p_b$ that is not $p_a$, and let $\ell'$ be a leaf at the end of a tree path for $g_b$. Since $\cN$ is normal, $\ell'\neq a$, and so $b\ell'|a\in \cR$. This last contradiction implies that $\{a, b\}$ is a cherry, thereby completing the proof of the lemma. \qed
\end{proof}

We next consider the recognition of reticulated cherries. For the purposes of establishing Theorem~\ref{main}, in addition to recognising a reticulated cherry, say $\{a, b\}$ in which $b$ is the reticulation leaf, we also want to determine the visibility set of $g_b$. To this end, we next introduce the notion of a candidate set.

Let $\cN$ be a phylogenetic network on $X$, and let $\cR$ and $\cQ$ be the sets of triples and quads displayed by $\cN$, respectively. Let $\{a, b\}\subseteq X$. A {\em candidate set for $b$} is a non-empty subset, $W_b$ say, of $X-\{a, b\}$ satisfying the following properties:
\begin{enumerate}[{\rm (I)}]
\item For all $c\in W_b$ and all $x\in X-(W_b\cup \{b\})$, the triple $bc|x\in \cR$, but the triple $ac|b\not\in \cR$.

\item For all distinct $c, c'\in W_b$, the triple $bc|c'\not\in \cR$.

\item For all $c\in W_b$, there is no $x\in X-(W_b\cup \{a, b\})$ such that $(x, b, c, a)\in \cQ$.
\end{enumerate}

\begin{lemma}
Let $\cN$ be a normal network on $X$, where $|X|\ge3$, and let $\cR$ and $\cQ$ be the sets of triples and quads displayed by $\cN$, respectively. Let $\{a, b\}\subseteq X$, and let $W_b$ be a candidate set for $b$. Then $\{a, b\}$ is a reticulated cherry of $\cN$ in which $b$ is the reticulation leaf and $W_b$ is the visibility set of $g_b$ if and only if $a$, $b$, and $W_b$ satisfy the following properties:
\begin{enumerate}[{\rm (i)}]
\item For all $x\in X-\{a, b\}$, the triple $ab|x\in \cR$.

\item For all $c\in W_b$, there is no $x\in X-(\{a, b\}\cup W_b)$ such that $(x, b, a, c)$ or $(x, a, b, c)$ is in $\cQ$.

\item If there exists an $x\in X-(\{a, b\}\cup W_b)$ such that $ac|x\in \cR$, where $c\in W_b$, then $(a, b, c, x)$ and $(c, b, a, x)$ are in $\cQ$.
\end{enumerate}
\label{cherry2}
\end{lemma}

\begin{proof}
Suppose that $\{a, b\}$ is a reticulated cherry of $\cN$ in which $b$ is the reticulation leaf and $W_b$ is the visibility set of $g_b$. It is easily seen that $a$, $b$, and $W_b=V_{g_b}$ satisfy (i) and (ii). To see that $a$, $b$, and $V_{g_b}$ satisfy (iii), assume that there is an $x\in X-(\{a, b\}\cup V_{g_b})$ such that $ac|x\in \cR$. Let $\cS$ be an embedding of $ac|x$ in $\cN$. Then $\cS$ uses $(p_a, a)$ as well as the arc directed into $g_b$ and the arc directed out of $g_b$ that is not $(g_b, p_b)$. Hence, by adjoining the arcs $(p_a, p_b)$ and $(p_b, b)$ to $\cS$, we construct an embedding of $(a, b, c, x)$ in $\cN$ and, by adjoining the arcs $(g_b, p_b)$ and $(p_b, b)$ to $\cS$, we construct an embedding of $(c, b, a, x)$ in $\cN$. Thus $(a, b, c, x), (c, b, a, x)\in \cQ$.

For the converse, suppose that $a$, $b$, and $W_b$ satisfy (i), (ii), and (iii). We first show that $p_b$ is a reticulation. Assume that $p_b$ is either a tree vertex or the root of $\cN$. Let $v$ denote the child of $p_b$ that is not $b$, and let $\ell$ be the leaf at the end of a tree path for $v$. If $v=\ell$, then, as $|X|\ge 3$, it follows by (i) that $\ell=a$; otherwise, $ab|\ell\not\in \cR$. But then $bc|a\not\in \cR$ for any $c\in W_b$, contradicting (I). Moreover, if $v$ is a tree vertex, then, as $ac|b\not\in \cR$ for each $c\in W_b$, at most one of $a$ and $c$ is an element of $C_v$. If $a\not\in C_v$, then, it is easily checked that $ab|\ell\not\in \cR$, contradicting (i). Thus $a\in C_v$ and so $c\not\in C_v$ for each $c\in W_b$. But then $bc|\ell\not\in \cR$ for any $c$, contradicting (I) in the choice of $W_b$. Therefore $v$ is a reticulation. Consider the cluster set $C_v$ of $v$. If $a\in C_v$ and $C_v\cap W_b\neq \emptyset$, then $ac|b\in \cR$, where $c\in C_v\cap W_b$, contradicting (I) in the choice of $W_b$. Furthermore, if $a\not\in C_v$ and $C_v\cap W_b=\emptyset$, then we can extend an embedding of the triple $ab|c$ in $\cN$, where $c\in W_b$, to an embedding of the caterpillar $(\ell, b, a, c)$, and so $(\ell, b, a, c)\in \cQ$, contradicting (ii). Thus exactly one of $a\in C_v$ and $C_v\cap W_b\neq \emptyset$ holds.

If $a\in C_v$ and $|C_v|\ge 2$, then, as $C_v\cap W_b$ is empty, $\cQ$ contains a caterpillar of the form $(x, a, b, c)$, where $x\in C_v-\{a\}$ and $c\in W_b$. This contradiction to (ii) implies that $|C_v|=1$, and so $C_v=\{a\}$. On the other hand, if $C_v\cap W_b\neq \emptyset$, then, as $W_b$ satisfies (II), it follows that $W_b$ is a subset of the visibility set $V_v$ of $v$. If $V_v-W_b$ is non-empty, then $bc|x\not\in \cR$, where $c\in W_b$ and $x\in V_v-W_b$, contradicting (I) in the choice of $W_b$. Thus, if $C_v\cap W_b\neq \emptyset$, then $V_v=W_b$.

Let $u$ denote the parent of $v$ that is not $p_b$. Since $\cN$ is normal, $(u, v)$ is not a shortcut and there is a tree path from $u$ to a leaf, $\ell'$ say. Let $P$ denote the arc set of this tree path. First assume that $C_v=\{a\}$. If $W_b\cap C_u\neq \emptyset$, then there is a triple $ac|b\in \cR$, where $c\in W_b$, contradicting (I) in the choice of $W_b$. Therefore $W_b\cap C_u$ is empty. Let $c\in W_b$. Since $W_b$ satisfies (I), $bc|a\in \cR$. If $\cS$ is an embedding of $bc|a$ in $\cN$, then $\cS$ uses $(u, v)$. It is now easily checked that the set of arcs
$$(\cS-\{(p_b, b), (u, v)\})\cup (P\cup \{(p_b, v)\})$$
are the arcs of an embedding of $ac|\ell'$ in $\cN$. Since $\ell'\in X-(\{a, b\}\cup W_b)$, it follows from (iii) that $(c, b, a, \ell')\in \cQ$. This is not possible as any phylogenetic tree displayed by $\cN$ with leaf set $\{a, b, c, \ell'\}$ in which $\{b, c\}$ is a cherry, also has $\{a, \ell'\}$ as a cherry. This contradiction implies that if $C_v=\{a\}$, then $p_b$ is a reticulation.

Second assume that $V_v=W_b$. If $a\in C_u$, then there is a triple $ac|b\in \cR$, where $c\in W_b$, contradicting (I) in the choice of $W_b$. Thus $a\not\in C_u$. Let $c\in W_b$. By (i), we have $ab|c\in \cR$. If $\cS$ is an embedding of $ab|c$ in $\cN$, then $\cS$ uses $(u, v)$, and so
$$(\cS-\{(p_b, b), (u, v)\})\cup (P\cup \{(p_b, v)\})$$
are the arcs of an embedding of $ac|\ell'$ in $\cN$. By (iii), $(a, b, c, \ell')\in \cQ$. But again this is not possible as any phylogenetic tree in $\cQ$ with leaf set $\{a, b, c, \ell'\}$ in which $\{a, b\}$ is a cherry, also has $\{c, \ell'\}$ as a cherry. Hence we have now established that $p_b$ is a reticulation.

Let $u_1$ and $u_2$ denote the parents of $p_b$, and let $\ell_1$ and $\ell_2$ denote the leaves at the end of tree paths for $u_1$ and $u_2$, respectively. Note that, since $W_b$ satisfies (II), if there exists an element $c$ in $W_b$ such that $c\in C_{u_i}$ for some $i\in \{1, 2\}$, then $W_b\subseteq V_{u_i}$. In turn, this implies that $W_b=V_{u_i}$; otherwise, $bc|x\not\in \cR$, where $c\in W_b$ and $x\in V_{u_i}-W_b$, contradicting (I). If $W_b\cap (C_{u_1}\cup C_{u_2})$ is empty, then, by considering an embedding of $bc|a$ in $\cN$, it is easily seen that either $(\ell_1, b, c, a)\in \cQ$ or $(\ell_2, b, c, a)\in \cQ$, contradicting (III) in the choice of $W_b$. Without loss of generality we may therefore assume that the visibility set of $u_2$ is $W_b$. If $a\not\in C_{u_1}$, then, by considering an embedding of $ab|c$ in $\cN$ for some $c\in W_b$, we deduce that $(\ell_1, b, a, c)\in \cQ$, contradicting (ii). Thus $a\in C_{u_1}$. If $|C_{u_1}|\ge 2$, then $(x, a, b, c)\in \cQ$, where $x\in C_{u_1}-\{a\}$, contradicting (ii). Hence $C_{u_1}=\{a\}$. We conclude that $\{a, b\}$ is a reticulated cherry of $\cN$ in which $b$ is the reticulation leaf and $W_b$ is the visibility set of $g_b$. \qed
\end{proof}

\subsection{Finding a candidate set} The algorithm associated with Theorem~\ref{main} involves finding the candidate sets for one of two leaves of a potential reticulated cherry. In this subsection, we consider how this can be done in polynomial time.

Let $\cN$ be a phylogenetic network on $X$, where $|X|\ge 3$, and let $\cR$ and $\cQ$ be the sets of triples and quads displayed by $\cN$, respectively. Let $\{a, b\}\subseteq X$, and suppose we want to find all candidate sets for $b$ if such a set exists or determine that there are no such sets. Potentially, we may have to consider all subsets of $X-\{a, b\}$. However, if $c\in X-\{a, b\}$, the next lemma shows that a candidate set for $b$ containing $c$, if it exists, is unique. Thus the task reduces to finding, for each $c\in X-\{a, b\}$, the candidate set for $b$ containing $c$ or determining that no such set exists.

\begin{lemma}
Let $\cR$ and $\cQ$ be the sets of triples and quads, respectively, displayed by a phylogenetic network $\cN$ on $X$, where $|X|\ge 3$. Let $\{a, b\}\subseteq X$ and let $c\in X-\{a, b\}$. If $W_b$ is a candidate set for $b$ containing $c$, then it is the unique candidate set for $b$ containing $c$.
\label{candidate1}
\end{lemma}

\begin{proof}
Let $W'_b$ be a candidate set for $b$ containing $c$, and let $x\in X-\{a, b, c\}$. Then, as $W'_b$ satisfies (I) and (II) in the definition of a candidate set, $x\in W'_b$ if and only if $bc|x\not\in \cR$. It follows that $W_b=W'_b$, that is, $W_b$ is the unique candidate set for $b$ containing $c$. \qed
\end{proof}

Called {\sc Candidate Set}, the following algorithm takes as its input $X$, $\cR$, $\cQ$, $\{a, b\}$, and $c$ and either finds a candidate set for $b$ containing $c$ or determines that there is no such set.
\begin{enumerate}[1.]
\item Set $U=\{x\in X-\{b, c\}: bc|x\in \cR\}$.

\item Set $W_b=(X-(U\cup \{b\}))\cup \{c\}$.

\item If $a$, $b$, and $W_b$ satisfy their namesakes in (I), (II), and (III), then return $W_b$.

\item Else, return {\em no candidate set for $b$ containing $c$}.
\end{enumerate}

The next lemma establishes the correctness and running time of {\sc Candidate Set}.

\begin{lemma}
Let $\cN$ be a phylogenetic network on $X$, where $|X|\ge 3$, and let $\cR$ and $\cQ$ be the sets of triples and quads displayed by $\cN$, respectively. Let $\{a, b\}\subseteq X$ and $c\in X-\{a, b\}$. Then {\sc Candidate Set} applied to $X$, $\cR$, $\cQ$, $\{a, b\}$, and $c$ correctly returns a candidate set for $b$ containing $c$ if it exists or the statement {\em no candidate set for $b$ containing $c$} if none exists. Furthermore, this application runs in time $O(|X|^2|\cR|+|X|^2|\cQ|)$, that is, $O(|X|^6)$.
\label{candidate2}
\end{lemma}

\begin{proof}
If {\sc Candidate Set} returns a set $W_b$, then, by Step~3, it is a candidate set for $b$ containing $c$. Conversely, if $W'_b$ is a candidate set for $b$ containing $c$. Then, by Lemma~\ref{candidate1}, $W'_b$ is the unique such set and so, by construction, at the end of Step~2, {\sc Candidate Set} constructs $W'_b$. It follows that {\sc Candidate Set} correctly returns $W'_b$.

For the running time, Steps~1 and~2 take $O(|X||\cR|)$ and $O(1)$ time, respectively, while Step~3 takes $O(|X|^2|\cR|+|X|^2|\cR|+|X|^2|\cQ|)$ time. Thus {\sc Candidate Set} completes in $O(|X|^2|\cR|+|X|^2|\cQ|)$ time, that is, in $O(|X|^6)$ time as $|\cR|\le |X|^3$ and $|\cQ|\le |X|^4$. \qed
\end{proof}

\section{Proof of Theorem~\ref{main}}
\label{proof}

In this section, we establish Theorem~\ref{main}. We start with two lemmas and the description of an operation on networks that underlies the induction in the proof of this theorem. Let $\cN$ be a phylogenetic network on $X$, and let $\{a, b\}$ be a subset of $X$. Suppose that $\{a, b\}$ is either a cherry or a reticulated cherry in which $b$ is the reticulation leaf. If $\{a, b\}$ is a cherry, then {\em deleting b} is the operation of deleting $b$ and its incident arc, and suppressing $p_a$ while, if $\{a, b\}$ is a reticulated cherry, then {\em deleting $b$} is the operation of deleting $b$, $p_b$, and their incident arcs, and suppressing $p_a$ and $g_b$. Note that the latter operation of deleting $b$ can be viewed as deleting $(g_b, p_b)$ and suppressing the resulting two degree-two vertices followed by the deletion of $b$ in the resulting network that has cherry $\{a, b\}$. The next lemma, which we will freely use throughout the rest of the paper, is now an immediate consequence of~\cite[Lemma 3.2]{bordewich18}.

\begin{lemma}
Let $\cN$ be a normal network on $X$, and let $\{a, b\}\subseteq X$, where $\{a, b\}$ is either a cherry or a reticulated cherry in which $b$ is the reticulation leaf. If $\cN'$ is obtained from $\cN$ by deleting $b$, then $\cN'$ is a normal network on $X-\{b\}$.
\label{deleting}
\end{lemma}

Let $\cR$ and $\cQ$ be the sets of triples and quads displayed by a phylogenetic network $\cN$ on $X$, respectively. Let $b\in X$, and let
$$\cR'=\{xy|z\in \cR: b\not\in \{x, y, z\}\}$$
and
$$\cQ'=\{(w, x, y, z)\in \cQ: b\not\in \{w, x, y, z\}\}.$$
We say that $\cR'$ and $\cQ'$ have been obtained from $\cR$ and $\cQ$, respectively, by {\em deleting $b$}. The proof of the next lemma is elementary and omitted.

\begin{lemma}
Let $\cN$ be a normal network on $X$, and let $\{a, b\}\subseteq X$, where $\{a, b\}$ is either a cherry or a reticulated cherry in which $b$ is the reticulation leaf. Furthermore, let $\cN'$ be the normal network obtained from $\cN$ by deleting $b$. If $\cR$ and $\cQ$ are the sets of triples and quads displayed by $\cN$, respectively, then the sets of triples and quads displayed by $\cN'$ are obtained from $\cR$ and $\cQ$ by deleting $b$.
\label{triples}
\end{lemma}

We now prove the uniqueness part of Theorem~\ref{main}.

\begin{proof}[Proof of Theorem~\ref{main}(i).]
The proof is by induction on the size of $X$. Since $\cN$ is normal, if $|X|=1$, then $\cN$ consists of an isolated vertex and, if $|X|=2$, then $\cN$ consist of two leaves adjoined to the root. In both cases, the theorem holds. 
Now suppose that $|X|\ge 3$, and that the theorem holds for all normal networks with at most $|X|-1$ leaves. Let $\cR$ and $\cQ$ be the sets of triples and quads, respectively, displayed by $\cN$. Let $\cN_1$ be a normal network on $X$ such that the sets of triples and quads displayed $\cN_1$ are $\cR$ and $\cQ$, respectively. By Lemma~\ref{cherry3}, $\cN$ has either a cherry, $\{a, b\}$ say, or a reticulated cherry, $\{a, b\}$ with reticulation leaf $b$ say.

First suppose that $\{a, b\}$ is a cherry of $\cN$. Then, by Lemma~\ref{cherry1}, $\{a, b\}$ is a cherry of $\cN_1$. Let $\cN'$ and $\cN'_1$ denote the normal networks obtained from $\cN$ and $\cN_1$, respectively, by deleting $b$. By Lemma~\ref{triples}, the sets of triples and quads of $\cN'$ and $\cN'_1$ are the same and so, by the induction assumption, up to isomorphism, $\cN'=\cN'_1$. Since $\{a, b\}$ is a cherry of both $\cN$ and $\cN_1$, it follows that, up to isomorphism, $\cN=\cN_1$. Thus, if $\{a, b\}$ is a cherry of $\cN$, part (i) of the theorem holds.

Now suppose that $\{a, b\}$ is a reticulated cherry of $\cN$ in which $b$ is the reticulation leaf and $V_{g_b}$ is the visibility set of $g_b$. Then, by Lemma~\ref{cherry2}, $\{a, b\}$ is a reticulated cherry of $\cN_1$ in which $b$ is the reticulation leaf. Furthermore, if $p'_b$ denotes the parent of $b$ in $\cN_1$, and $g'_b$ denotes the parent of $p'_b$ that is not the parent of $a$ in $\cN_1$, then, by the same lemma, $V_{g_b}$ is the visibility set of $g'_b$. 
Let $\cN'$ and $\cN'_1$ denote the normal networks obtained from $\cN$ and $\cN_1$, respectively, by deleting $b$. By Lemma~\ref{triples}, the sets of triples and quads of $\cN'$ and $\cN'_1$ coincide. Therefore, by the induction assumption, up to isomorphism, $\cN'=\cN'_1$. We complete the proof by showing that, by subdividing, in $\cN'$ (equivalently, $\cN'_1$) there is exactly one arc to insert $p_a$ to adjoin $(p_a, p_b)$ and exactly one arc to insert $g_b$ to adjoin $(g_b, p_b)$ so that, together with the arc $(p_b, b)$, the resulting network is normal, and displays $\cR$ and $\cQ$. We do this using only $\cR$, $\cQ$, and $V_{g_b}$.

Evidently, there is exactly one arc to adjoin $(p_a, p_b)$, namely the arc incident to $a$. Now consider the placement of $g_b$ in $\cN'$. Let $U$ be the subset of vertices of $\cN'$ consisting of each vertex $u$ having the property that $V_u=V_{g_b}$. Observe that in $\cN$, the child of $g_b$ that is not $p_b$ has this property. If $|U|\geq 2$, let $u$ and $u'$ be distinct vertices in $U$. Assume that there is no directed path from $u$ to $u'$ or from $u'$ to $u$. Then, if $\ell'$ is the leaf at the end of a tree path for $u'$, there is a directed path from the root of $\cN'$ to $\ell'$ avoiding $u$, and so $\ell'\not\in V_u$. But $\ell'\in V_{u'}$, a contradiction. Hence, for every pair of vertices in $U$, there is a directed path in $\cN'$ connecting the two vertices. Furthermore, as $\cN'$ is acyclic, there is a directed path from $u$ to $u'$ if and only if there is no directed path from $u'$ to $u$.

Let $P$ be a directed path from the root of $\cN'$ to a leaf $\ell$ such that no directed path from the root of $\cN$ to a leaf has more vertices in $U$ than $P$. Order the vertices in $U$ on $P$, say $u_1, u_2, \ldots, u_k$, so that it is consistent with $P$. That is, $u_i$ is before $u_j$ on $P$ precisely if $i < j$. If there is a vertex $v$ in $U$ that is not on $P$, then, by the maximality of $P$, for some $i\in \{1, 2, \ldots, k-1\}$, there is a directed path from each of $u_1, u_2, \ldots, u_i$ to $v$, and no such path from $u_{i+1}$ to $v$, but there is a directed path from $v$ to $u_{i+1}$. Since $\cN'$ is acyclic, it follows that we can construct a path in $\cN'$ by taking the subpath of $P$ from the root of $\cN'$ to $u_i$, and then adjoining a directed path from $u_i$ to $v$, a directed path from $v$ to $u_{i+1}$, and the subpath of $P$ from $u_{i+1}$ to $\ell$. This constructed path contradicts the maximality of $P$. Thus we may assume that the vertex set of $P$ contains $U$. Furthermore, as $\cN'$ is normal, we may also assume that the subpath of $P$ from $u_k$ to $\ell$ is a tree path.

Let $P'$ denote the subpath of $P$ from $u_1$ to $u_k$. We next show that every vertex in $P'$ is in $U$ and, except possibly $u_1$, every vertex in $P'$ is a tree vertex or a leaf. First assume that $w\neq u_1$ is a reticulation in $P'$. Without loss of generality, we may choose $w$ to be the reticulation in $P'$ closest to $u_1$. Let $v_1$ and $v_2$ denote the parents of $w$ in $\cN'$, and let $\ell_1$ and $\ell_2$ be leaves at the end of tree paths for $v_1$ and $v_2$, respectively. We may assume $v_1$ is in $P'$. Then $\ell_1\in V_{u_1}$ as either $u_1=v_1$ or there is no reticulation between $u_1$ and $v_1$ in $P'$ except possibly $u_1$. But $\ell_1\not\in V_{u_k}$, a contradiction. Thus every vertex in $P'$, except possibly $u_1$, is a tree vertex or a leaf. Now assume that there is a vertex $w$ on $P'$ that is not in $U$. Then, as the subpath of $P'$ from $w$ to $u_k$ consists of tree vertices, it follows that $V_{u_k}\subseteq V_w$. If $V_{u_k}= V_w$, then $V_w\in U$, a contradiction. Moreover, if $V_w-V_{u_k}$ is non-empty, then, as the subpath of $P'$ between $u_1$ and $w$ consists of tree vertices, $V_{u_k}\subsetneqq V_{u_1}$, again a contradiction. Thus every vertex on $P'$ is in $U$.

Now, let $\cI=\{1, 2, \ldots, k-1\}$ if $u_1$ is a tree vertex, and let $\cI=\{2, 3, \ldots, k-1\}$ otherwise. Furthermore, for all $i\in \cI$, let $v_i$ denote the child of $u_i$ that is not on $P$, and let $m_i$ denote the leaf at the end of a tree path for $v_i$. If, for some $i\in \cI$, the vertex $v_i$ is a tree vertex or a leaf, then it is easily checked that $m_i\in V_{u_i}$, but $m_i\not\in V_{u_k}$. This contradiction implies that $v_i$ is a reticulation for each $i\in\cI$. Clearly, the elements in $\{m_i: i\in\cI\}$ are pairwise distinct. Also, if $u_k\neq \ell$, let $u_k'$ be the child of $u_k$ that is not on $P$. If $u_k'$ is a reticulation, then the visibility set of the child of $u_k$ that is on $P$ is $V_{u_k}$, and therefore an element in $U$; a contradiction. Hence, if $u_k\neq \ell$, then $u_k'$  is a tree vertex or a leaf.

At last, we consider the placement of $g_b$ in $\cN'$. By the construction of $P$, the vertex $g_b$ corresponds to a subdivision of an arc directed into a vertex in $\{u_i: i\in \cI\cup \{k\}\}$. If $k=1$, then $u_1$ is a tree vertex or a leaf, and the unique placement of $g_b$ is a subdivision of the arc directed into $u_1$. If $k=2$ and $u_1$ is a reticulation, then $u_2$ is a tree vertex or a leaf, and the unique placement of $g_b$ is a subdivision of the arc directed into $u_2$. So assume that either $k=2$ and $u_1$ is a tree vertex, or $k\ge 3$. Let
$$\cQ_P=\{(m_i, \ell, b, a)\in \cQ: i\in \cI\}.$$
Furthermore, let $i'$ be the minimum element in $\cI$ for which $(m_{i'}, \ell, b, a)\in \cQ_P$. Then it is easily seen that each of
$$(m_{i'}, \ell, b, a), (m_{i'+1}, \ell, b, a), \ldots, (m_{k-1}, \ell, b, a)$$
is in $\cQ_P$. In particular, the unique placement of $g_b$ is a subdivision of the arc directed into $u_k$ if $\cQ_P$ is empty and $u_{i'}$ otherwise. This completes the proof of part~(i) of the theorem. \qed
\end{proof}

\subsection{The algorithm} Let $\cR$ and $\cQ$ be the sets of triples and quads, respectively, of a normal network $\cN$ on $X$. Called {\sc Construct Normal}, we now present a recursive algorithm whose input is $X$, $\cR$, and $\cQ$ and returns a normal network $\cN_0$ that is isomorphic to $\cN$. The correctness of the algorithm is essentially established in the constructive proof of Theorem~\ref{main}(i), and so it is omitted. The running time of {\sc Construct Normal} is given immediately after its description.

\begin{enumerate}[1.]
\item If $|X|=1$, then return the phylogenetic network consisting of the single vertex in $X$.

\item If $|X|=2$, then return the phylogenetic network consisting of the two leaves in $X$ adjoined to the root.

\item Else, find either $\{a, b\}\subseteq X$, or $\{a, b\}\subseteq X$ and $W_b\subseteq X-\{a, b\}$, where $W_b$ is a candidate set for $b$, satisfying the sufficiency conditions of their namesakes in the statements of Lemmas~\ref{cherry1} or~\ref{cherry2}, respectively.
\item Delete $b$ in $\cR$ and $\cQ$ to give the sets $\cR'$ and $\cQ'$ of triples and quads, respectively, on $X'=X-\{b\}$.
\begin{enumerate}[(a)]
\item If $\{a, b\}\subseteq X$ satisfies the sufficiency condition in Lemma~\ref{cherry1}, then apply {\sc Construct Normal} to input $X'$, $\cR'$, and $\cQ'$, construct $\cN_0$ from the returned normal network $\cN'_0$ on $X'$ by subdividing the arc directed into $a$ with a new vertex $p_a$,  adjoin a new leaf $b$ to $p_a$ via the new arc $(p_a, b)$, and return $\cN_0$.

\item Else, $\{a, b\}\subseteq X$ and $W_b\subseteq X-\{a, b\}$ satisfy the sufficiency conditions of Lemma~\ref{cherry2}. Apply {\sc Construct Normal} to input $X'$, $\cR'$, and $\cQ'$, and construct $\cN_0$ from the returned normal network $\cN'_0$ on $X'$ as follows.
\begin{enumerate}[(i)]
\item Find the vertex $u$ whose visibility set is $W_b$ and whose cluster set is minimal with respect to containing $W_b$.

\item Find the path $u_1, u_2, \ldots, u_k$ of vertices, where $u_k=u$, consisting of precisely those vertices in $\cN'_0$ whose visibility set is $W_b$. Let $\ell$ denote the leaf at the end of a tree path for $u_k$.

\item Let $\cI=\{1, 2, \ldots, k-1\}$ if $u_1$ is a tree vertex; otherwise, let $\cI=\{2, 3, \ldots, k-1\}$. For each $i\in \cI$, let $v_i$ denote the reticulation child of $u_i$, and let $m_i$ denote the leaf at the end of a tree path for $v_i$.

\item If $k=1$, subdivide the arc directed into $u_1$ with a new vertex $g_b$. If $k=2$ and $u_1$ is a reticulation, subdivide the arc directed into $u_2$ with a new vertex $g_b$.

\item Else, $k=2$ and $u_1$ is a tree vertex, or $k\ge 3$. If there is no quad $(m_i, \ell, b, a)$ in $\cQ$, where $i\in \cI$, subdivide the arc directed into $u_k$ with a new vertex $g_b$. Otherwise, subdivide the arc directed into $u_{i'}$, where $i'$ is the smallest $i$ such that $(m_{i'}, \ell, b, a)\in \cQ$, with a new vertex $g_b$.

\item Subdivide the arc directed into $a$ with a new vertex $p_a$, adjoin a new vertex $p_b$ via new arcs $(p_a, p_b)$ and $(g_b, p_b)$, adjoin a new leaf $b$ via a new arc $(p_b, b)$, and set $\cN_0$ to be the resulting network.

\item Return $\cN_0$.
\end{enumerate}
\end{enumerate}
\end{enumerate}

We now consider the running time of {\sc Construct Normal}.

\begin{proof}[Proof of Theorem~\ref{main}(ii).]
The algorithm takes as input a set $X$, and sets $\cR$ and $\cQ$ of triples and quads, respectively, of a normal network $\cN$ on $X$. If $|X|\in \{1, 2\}$, then the algorithm runs in constant time. If $|X|\ge 3$, then the algorithm begins by finding either $\{a, b\}\subseteq X$, or $\{a, b\}\subseteq X$ and $W_b\subseteq X-\{a, b\}$, where $W_b$ is a candidate set for $b$, satisfying the sufficiency conditions of their namesakes in Lemmas~\ref{cherry1} or~\ref{cherry2}, respectively. In the worst possible instance, the longest running part of this process involves finding the latter. The check that $a$, $b$, and $W_b$ satisfy the sufficiency conditions in Lemma~\ref{cherry2} takes $O(|X||\cR|+|X|^2|\cQ|+|X|^2|\cR||\cQ|)$ time. Since there is asymmetry between $a$ and $b$ and, by Lemma~\ref{candidate1}, there are at most $|X|$ candidate sets for $b$, the number of such checks is at most $O(|X|^3)$. Thus the total time to run the checks is
$$O(|X|^3(|X||\cR|+|X|^2|\cR||\cQ|)).$$
But we also need to find, for each leaf $b$, all candidate sets for $b$. By Lemma~\ref{candidate2}, as there are at most $|X|$ candidate sets for $b$, it takes
$$O(|X|(|X|^2|\cR|+|X|^2|\cQ|))$$
time to find all such sets for $b$. Thus the running time to complete Step~3 is $O(|X|^3(|X||\cR|+|X|^2|\cR||\cQ|))$.

We next delete $b$ in $\cR$ and $\cQ$, and this takes at most $O(|\cR|+|\cQ|)$ time. Clearly, Step~4(a) takes less time to complete than Step~4(b), so we may assume that the latter is reached. To determine the cluster set of a vertex $u$ of $\cN'_0$, a single postorder transversal of $\cN'_0$ is sufficient. Furthermore, to determine the visibility set of a vertex $u$ of $\cN'_0$, we delete $u$ and its incident arcs and check, for each leaf $\ell$ in $X'$ whether the resulting rooted acyclic digraph, $D'$ say, has a directed path from its root $\rho'$ to $\ell$. Effectively, we are finding the `cluster set' $C_{\rho'}$ of $\rho'$ in $D'$. The visibility set of $u$ in $\cN'_0$ consists precisely of those leaves in $X'$ not in $C_{\rho'}$. A single postorder transversal of $D'$ is sufficient to determine $C_{\rho'}$. Since $\cN$ has at most $O(|X|)$ vertices and, therefore, at most $O(|X|)$ edges in total~\cite{bic12} (also see~\cite{mcd15}), it takes $O(|X|)$ time to find the visibility set of $u$, and so it takes $O(|X|^2)$ time to complete Steps~4(b)(i) and~4(b)(ii). Once $u_1, u_2, \ldots, u_k$ are determined, finding the leaves $m_2, m_3, \ldots, m_{k-1}$, and possibly $m_1$ if $u_1$ is a tree vertex and $k\ge 2$, takes $O(|X|^2)$ time as $k\le |X|$. If performed, Step~4(b)(iv) takes constant time, while Step~4(b)(v) takes $O(|X||\cQ|)$ time. 
Thus the location of $g_b$, ignoring the running time of Step~3, can be found in time $O(|X|^2+|X||\cQ|)$. Since Step~4(b)(vi) takes constant time, it follows that Step~4 takes $O(|X|^2+|X||\cQ|)$ time to complete. Hence $\cN_0$ can be returned in
$$O(|X|^3(|X||\cR|+|X|^2|\cR||\cQ|))$$
time, and so the total time of each iteration is $O(|X|^3(|X||\cR|+|X|^2|\cR||\cQ|))$.

When recursing, the input to the recursive call is a set $X'$, and sets $\cR'$ and $\cQ'$ of triples and quads of a normal network on $|X|-1$ leaves. Therefore the total number of iterations is $O(|X|)$. Hence {\sc Construct Normal} completes in $O(|X|^4(|X||\cR|+|X|^2|\cR||\cQ|))$time, that is, in $O(|X|^{13})$ time as $|\cR|\le |X|^3$ and $|\cQ|\le |X|^4$. This completes the proof of Theorem~\ref{main}(ii). \qed
\end{proof}

\section{Temporal Normal Networks}

In this section, we consider a certain subclass of normal networks and briefly outline how they can be reconstructed from their sets of displayed triples and quads. Let $\cN$ be a phylogenetic network on $X$, and let $V$ be the vertex set of $\cN$. We say that $\cN$ is {\em temporal} if there exists a map $t: V\rightarrow\mathbb{R}^+$ such that for all $u, v\in V$, we have $t(u)=t(v)$ if $(u, v)$ is a reticulation arc and $t(u) < t(v)$ if $(u, v)$ is a tree arc. Note that the two networks shown in Fig.~\ref{fig:triples} are temporal and, so, temporal normal networks cannot be determined by their set of displayed triples. Biologically, if a phylogenetic network is temporal, then it satisfies two natural timing constraints. Firstly, speciation events occur successively and, secondly, reticulation events occur contemporaneously and so such events are realised by coexisting ancestral species.

Let $\cN$ be a phylogenetic network on $X$, and let $\{a, b, c\}$ be a three-element subset of $X$. If $p_b$ is a reticulation, and both $\{a, b\}$ and $\{b, c\}$ are reticulated cherries, then $\{a, b, c\}$ is referred to as a {\em double-reticulated cherry} of $\cN$ in which $b$ is the {\em reticulation leaf}. Now, for a temporal normal network $\cN$, let $u$ be a tree vertex such that, for all other tree vertices $u'$, we have $t(u)\ge t(u')$, it is straightforward to show that $\cN$ has either a cherry or a double-reticulated cherry (see, for example, \cite{bor16}). For a double-reticulated cherry $\{a, b, c\}$ in $\cN$ in which $b$ is the reticulation leaf, consider the operation of deleting $b$ (as defined for a more general reticulated cherry in Section~\ref{proof}). Recall that this operation corresponds to deleting $b$, $p_b$, and their incident arcs, and suppressing the resulting degree-two vertices. 
Noting that a temporal tree-child network is normal, it  follows from~\cite[Lemma 5.1]{bor16} that the phylogenetic network obtained from $\cN$ by deleting $b$ is temporal and normal. Analogous to Lemma~\ref{cherry2}, the next lemma establishes necessary and sufficient conditions to recognise a double-reticulated cherry in a normal network.


\begin{lemma}
Let $\cN$ be a normal network on $X$, where $|X|\ge 3$, and let $\cR$ and $\cQ$ be the sets of triples and quads displayed by $\cN$, respectively. Let $\{a, b, c\}\subseteq X$. Then $\{a, b, c\}$ is a double-reticulated cherry of $\cN$ in which $b$ is the reticulation leaf if and only if $\{a, b, c\}$ satisfies the following three properties:
\begin{enumerate}[{\rm (i)}]
\item For all $x\in X-\{a, b\}$, the triple $ab|x\in \cR$ and, for all $x\in X-\{b, c\}$, the triple $bc|x\in \cR$, but the triple $ac|b\not\in \cR$.

\item There is no $x\in X-\{a, b, c\}$ such that $(x, b, a, c)$, $(x, b, c, a)$, $(x, a, b, c)$, or $(x, c, b, a)$ is in $\cQ$.

\item If there exists an $x\in X-\{a, b, c\}$ such that $ac|x\in \cR$, then $(a, b, c, x)$ and $(c, b, a, x)$ are in $\cQ$.
\end{enumerate}
\label{double-cherry}
\end{lemma}

We omit the proof as it is a consequence of Lemma~\ref{cherry2}. In particular, in viewing a double-reticulated cherry $\{a, b, c\}$ with reticulation leaf $b$ as a reticulated cherry $\{a, b\}$ with reticulation leaf $b$, observe that the visibility set of $g_b$ is $\{c\}$. Of course, the same observation applies to $\{b, c\}$ but with the roles of $a$ and $c$ interchanged. We now turn back to temporal normal networks since they are phylogenetic networks for which we can repeatedly delete the reticulation leaf of a double-reticulated cherry or a leaf of a cherry until we are left with a single vertex. Hence, with Lemma~\ref{double-cherry} in hand, we can reconstruct a temporal normal network from its sets of displayed triples and quads using an algorithm that is a simplification of {\sc Construct Normal}. Without going into details, Steps~3 and~4(b) of {\sc Construct Normal} can be simplified in the following way, while the other steps remain unchanged. If the input  is a temporal normal network $\cN$ on $X$ as well as its sets of displayed triples and quads, Step~3 finds $\{a, b\}\subseteq X$ or $\{a, b, c\}\subseteq X$ that satisfies the conditions of their namesakes in the statements of Lemmas~\ref{cherry1} or~\ref{double-cherry}. Moreover, Step~4(b) reconstructs a temporal normal network on $X$ from a temporal normal network on $X-\{b\}$ by subdividing the arc directed into $a$ (resp.\ $c$) with a new vertex $p_a$ (resp.\ $p_c$), adjoining a new vertex $p_b$ via new arcs $(p_a, p_b)$ and $(p_c, p_b)$, and adjoining a new leaf $b$ via a new arc $(p_b, b)$.

\end{document}